\documentclass[12pt]{amsart}
\usepackage{amsmath,amssymb,amsthm}
\usepackage{latexsym}      
\usepackage{color,colordvi,graphicx}
\numberwithin{equation}{section}
\newtheorem{theorem}{Theorem}
\newtheorem{proposition}[theorem]{Proposition}

\theoremstyle{remark}

\newtheorem{example}[theorem]{Example}
\newtheorem{remark}[theorem]{Remark}

\newcounter{FNC}[page]
\def\fauxfootnote#1{{\addtocounter{FNC}{2}$^\fnsymbol{FNC}$%
     \let\thefootnote\relax\footnotetext{$^\fnsymbol{FNC}$\Magenta{#1}}}}
\newcommand{\defcolor}[1]{\RoyalBlue{#1}}
\newcommand{\demph}[1]{\defcolor{{\sl #1}}}

\newcommand{\frakS}{{\mathfrak S}}
\newcommand{\calF}{{\mathcal F}}

\newcommand{\Fln}{{\mathbb F}\ell(n)}
\newcommand{\Flan}{{\mathbb F}(\adot;n)}
\newcommand{\Gr}{\mbox{\rm Gr}}
\newcommand{\ch}{\mbox{\rm ch}}
\newcommand{\adot}{{a_\bullet}}
\newcommand{\Fdot}{{F_\bullet}}
\newcommand{\Edot}{{E_\bullet}}

\newcommand{\C}{{\mathbb C}}

\newcommand{\Q}{{\mathbb Q}}

\DeclareMathOperator{\up}{up}
\DeclareMathOperator{\rk}{rk}

\DeclareMathOperator{\het}{ht}
\DeclareMathOperator{\ep}{end}

\title{Two Murnaghan-Nakayama Rules in Schubert Calculus}

\author{Andrew Morrison}
\address{Andrew Morrison \\
         Department of Mathematics\\
         ETH,
         Z\"urich, Switzerland}
\email{andrewmo@math.ethz.ch}
\urladdr{http://www.math.ethz.ch/~andrewmo/}
\author{Frank Sottile}
\address{Frank Sottile \\
         Department of Mathematics\\
         Texas A\&M University\\
         College Station\\
         Texas \ 77843\\
         USA}
\email{sottile@math.tamu.edu}
\urladdr{http://www.math.tamu.edu/~sottile}
\thanks{Research of Morrison supported by the Swiss National Science Foundation through
  grant SNF-200021-143274 and the MSRI} 
\thanks{Research of Sottile supported by the NSF through grants DMS-1001615 and DMS-1501370 and the MSRI}
\begin{document}
%%%%%%%%%%%%%%%%%%%%%%%%%%%%%%%%%%%%%%%%%%%%%%%%%%%%%%%%%%%%%%%%%%%%%%%%%%

\begin{abstract}
 The Murnaghan-Nakayama rule expresses the product of a Schur function with a Newton power sum in the basis of
 Schur functions. 
 We establish a version of the Murnaghan-Nakayama rule for Schubert polynomials and a version for the 
 quantum cohomology ring of the Grassmannian.
 These rules compute all intersections of Schubert cycles with tautological classes coming from the
 Chern character. 
 Like the classical rule, both rules are multiplicity-free signed sums.
\end{abstract}

\maketitle

%%%%%%%%%%%%%%%%%%%%%%%%%%%%%%%%%%%%%%%%%%%%%%%%%%%%%%%%%%%%%%%%%%%%%%%%%%%%%%%%%%%%%%%%%%
%
\section{Introduction}
Each integer partition $\lambda$ has an associated Schur symmetric function $s_\lambda$,
and these Schur functions form a basis for the $\Q$-algebra of symmetric functions, which
is also freely generated by the power sum symmetric functions $p_r$.
The Murnaghan-Nakayama rule is the expansion in the Schur
basis of the product by a power sum,
 \begin{equation*}
   p_r\cdot s_\lambda\ =\ 
   \sum_{\mu} (-1)^{\het(\mu/\lambda)+1} s_\mu\,,
 \end{equation*}
the sum over all partitions $\mu$ such that $\mu/\lambda$ is a rim hook of size $r$ and $\het(\mu/\lambda)$ is
the height (number of rows) of $\mu/\lambda$.

Products $p_\lambda:=p_{\lambda_1}\dotsb p_{\lambda_k}$ of power sums form another basis for
symmetric functions, and the change of basis matrix between these two is the character table for the 
symmetric group.
In this way, the Murnaghan-Nakayama rule gives a formula for the characters of the symmetric
group~\cite{Mu37,Na41a,Na41b}. 

The cohomology ring of the Grassmannian $\Gr(k,n)$ of $k$-planes in $n$-space has a basis
of Schubert cycles $\sigma_\lambda$ for $\lambda$ a partition with $\lambda_1\leq n{-}k$
and $\lambda_{k+1}=0$ (written $\lambda\leq\Box_{k,n}$).
Its multiplication is induced from the ring of symmetric functions under the map that sends a Schur function
$s_\lambda$ to $\sigma_\lambda$ when $\lambda\leq\Box_{k,n}$ and otherwise sends it to $0$.
Images of the power sum functions give the Chern characters of the tautological bundle.
Thus the Murnaghan-Nakayama rule computes intersections of Schubert cycles with tautological
classes. 

We extend this classical formula in two directions, to Schubert polynomials~\cite{Mac91}
and to the quantum cohomology of Grassmannians~\cite{Be97}.
For each permutation $w\in S_n$, let $\frakS_w$ be the Schubert polynomial of Lascoux and
Sch\"utzenberger~\cite{LS82}. 

%%%%%%%%%%%%%%%%%%%%%%%%%%%%%%%%%%%%%%%%%%%%%%%%%%%%%%%%%%%%%%%%%%%%%%%%%%%%%%%%%
\begin{theorem}\label{Th:one}
 Let $k,r<n$ be positive integers and $w\in S_n$ a permutation.
 Then
\[
    p_r(x_1,\dotsc,x_k)\cdot\frakS_w\ =\ 
   \sum (-1)^{\het(\eta)+1}\frakS_{w\eta}\,,
\]
 the sum over all $(r{+}1)$-cycles $\eta$ such that 
\[
    w\ <_k\ w\eta 
    \qquad\mbox{with}\qquad
    \ell(w\eta)=\ell(w)+r\,,
\]
 where $<_k$ is the $k$-Bruhat order and $\het(\eta)= \#\{i\leq k\mid \eta(i)\neq i\}$.
\end{theorem}
%%%%%%%%%%%%%%%%%%%%%%%%%%%%%%%%%%%%%%%%%%%%%%%%%%%%%%%%%%%%%%%%%%%%%%%%%%%%%%%%%

For simplicity, our notation intentionally suppresses the dependence of $\het(\eta)$ on $k$.

The cohomology of the manifold $\Fln$ of complete flags in $\C^n$ has a basis of Schubert cycles $[X_w]$
indexed by permutations $w$ of $1,\dotsc,n$.
Mapping the variables $x_1,\dotsc,x_n$ to the Chern roots of the dual of the tautological
flag bundle is a surjection onto the cohomology ring with the Schubert polynomial
$\frakS_w$ sent to the Schubert cycle $[X_w]$ and power sums sent to tautological classes.
Thus this Murnaghan-Nakayama rule also computes intersections of Schubert cycles with tautological classes.

Additively, the quantum cohomology $qH^*(\Gr(k,n))$ of the Grassmannian $\Gr(k,n)$ is the vector space of
polynomials in a parameter $q$ with coefficients from $H^*(\Gr(k,n))$.
It has a new multiplication $*$ among Schubert cycles that encodes the three-point Gromov-Witten invariants
and reduces to the usual one when $q=0$~\cite{Be97}. 

%%%%%%%%%%%%%%%%%%%%%%%%%%%%%%%%%%%%%%%%%%%%%%%%%%%%%%%%%%%%%%%%%%%%%%%%%%%%%%%%%
\begin{theorem}\label{Th:two}
 Let $k,r< n$ be positive integers and $\lambda\leq \Box_{k,n}$.
 Then
\[
    p_r * \sigma_\lambda\ =\ 
      \sum_\mu  (-1)^{\het(\mu/\lambda)+1} \sigma_\mu 
     \ -\ (-1)^k q \sum_\nu  (-1)^{\het(\lambda/\nu)+1} \sigma_\nu\,,
\]
 where the first sum is over all $\mu\leq\Box_{k,n}$ with $\mu/\lambda$ a rim hook of size $r$ and the second
 sum is over all $\nu\leq\lambda$ with $\lambda/\nu$ a rim hook of size $n{-}r$.
\end{theorem}
%%%%%%%%%%%%%%%%%%%%%%%%%%%%%%%%%%%%%%%%%%%%%%%%%%%%%%%%%%%%%%%%%%%%%%%%%%%%%%%%%

The Murnaghan-Nakayama rule has many generalizations.
Fomin and Green gave a version for non-commutative symmetric functions, which led to formulas
for characters of representations associated to stable Schubert and Grothendieck polynomials~\cite{FoGr98}. 
McNamara gave a skew version~\cite{McN}, which Konvalinka generalized to a skew
rule for multiplication by a `quantum' (perturbed by a parameter $q$) power sum function~\cite{Ko12}. 
Bandlow, et al.\ gave a version in the cohomology of an affine Grassmannian~\cite{BSZ11}. 
Tewari gave a version for noncommutative Schur functions~\cite{Tewari}.
Wildon gave a plethystic version~\cite{Wildon}.
Ross gave a version for loop Schur functions~\cite{Ro14}, providing a fundamental step in
the orbifold Gromov--Witten/Donaldson--Thomas correspondence~\cite{RoZo13}.  

This paper is organized as follows.
In Section~\ref{S:classical} we recall some aspects of the Schur basis of symmetric functions and 
the classical Murnaghan-Nakayama rule.
In Section~\ref{S:flag}, we recall some results concerning the Schubert polynomials, prove Theorem~\ref{Th:one},
and explain its significance for the flag manifold.
We conclude with Section~\ref{S:quantum}, where we recall the quantum cohomology ring of the Grassmannian and
prove Theorem~\ref{Th:two}.

Both authors thank the MSRI where this work began in Winter 2013.
This paper is an expanded version of an abstract written by Morrison for the 2014 FPSAC
conference~\cite{AM}. 

%%%%%%%%%%%%%%%%%%%%%%%%%%%%%%%%%%%%%%%%%%%%%%%%%%%%%%%%%%%%%%%%%%%%%%%%%%
%
\section{The classical Murnaghan-Nakayama rule}\label{S:classical}

For a complete and elegant treatment of symmetric functions see Macdonald's book~\cite{Ma95}.
The $\Q$-algebra \demph{$\Lambda$} of homogeneous symmetric functions in countably many
indeterminants $x_1,x_2,\dotsc$ is freely generated by three distinguished sequences of
symmetric functions.
One consists of the elementary symmetric functions, $e_a$, where
$e_a$ is the formal sum of all square-free monomials of degree $a$.
A second consists of the complete symmetric functions, $h_b$, where $h_b$ is the formal sum of all monomials 
of degree $b$. 
The third distinguished sequence consists of the power sum functions,
$p_r$, where $p_r$ is the formal sum of the $r$th powers of the variables, 
$p_r=x_1^r+x_2^r+\dotsb$.

The most important family of symmetric functions are the Schur functions $s_\lambda$,
which are indexed by partitions 
$\lambda\colon \lambda_1\geq\dotsb\geq\lambda_k\geq 0$,
where $\lambda_i$ is an integer.
The Schur function $s_\lambda$ may be defined by the Jacobi-Trudi formula as a determinant
\[
   s_\lambda\ =\ \det\left( h_{\lambda_i+j-i}\right)_{i,j=1}^k\,.
\]
%is the generating function for semi-standard Young tableaux of shape $\lambda$,
%\[
%  s_\lambda\ =\ \sum_T x^T\,,
%\]
%where $x^T$ is a monomial whose exponent records the entries in $T$.
This has degree $\defcolor{|\lambda|}:=\lambda_1+\dotsb+\lambda_k$.
We have that $s_{(b)}=h_b$ and $s_{(1^a)}=e_a$, where $(1^a)$ is a sequence of $a$ 1s.
The set of all Schur functions forms a basis for $\Lambda$.

We often represent a partition $\lambda$ by its Young diagram, which is a left-justified
array of boxes with $\lambda_i$ boxes in row $i$.
For example
\[
   (3,1,0,0)\ \longleftrightarrow\ 
   \raisebox{-3.5pt}{\includegraphics{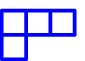}}
      \qquad\mbox{and}\qquad
   (5,4,3,1)\ \longleftrightarrow\ 
   \raisebox{-10.5pt}{\includegraphics{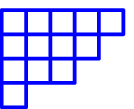}}\;.
\]
When $\lambda$ and $\mu$ are partitions with $\lambda_i\leq\mu_i$ for all $i$, we write 
\defcolor{$\lambda\leq\mu$} and write \defcolor{$\mu/\lambda$} for the set-theoretic
difference  $\mu\smallsetminus\lambda$ of their diagrams
The \demph{size} of $\lambda/\mu$ is its number of boxes.

A \demph{rim hook} is a skew shape $\mu/\lambda$ that meets a connected set of Northwest to Southeast
($\searrow$) diagonals with no two boxes in the same  diagonal.
The \demph{height $\het(\mu/\lambda)$} of a rim hook $\mu/\lambda$ is its
number of rows. 
A \demph{horizontal strip} $\lambda/\mu$  has no two boxes in the same column, and a \demph{vertical strip}
$\lambda/\mu$ has no two boxes in the same row.
Below are rim hooks of height two, three, and four of sizes five, six, and seven, respectively, as well as a
horizontal strip of size five and a vertical strip of size four.
\[
  \includegraphics{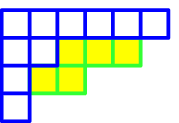}\qquad
  \includegraphics{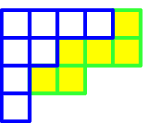}\qquad
  \includegraphics{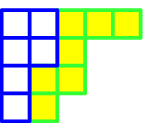}\qquad
  \includegraphics{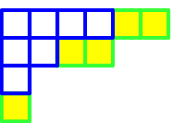}\qquad
  \includegraphics{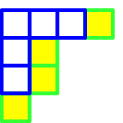}
\]

A partition $\lambda$ is a \demph{hook} if $\lambda_2\leq 1$.
For $a,b\geq 1$, the hook with $a$ rows and $b$ columns is $(b,1^{a-1})$.
Since $e_{a-1}\cdot h_{b}=s_{(b,1^{a-1})}+s_{(b{+}1,1^{a-2})}$, we have 
 \begin{equation}\label{Eq:HookFormula}
    s_{(b,1^{a-1})}\ =\ e_{a-1}\cdot h_b - e_{a-2}\cdot h_{b+1} + \dotsb + 
                     (-1)^{a-1} h_{a+b-1}\,.
 \end{equation}
The power sum symmetric function $p_r$ is the alternating sum of hooks of size $r$,
 \begin{equation}\label{Eq:hookSum}
 p_r\ =\  h_r-s_{(r-1,1)}+\dotsb+(-1)^{r-1}e_r\ =\ 
        \sum_{i=0}^{r-1}(-1)^i s_{(r-i,1^i)}\,.
 \end{equation}
We give formulas for the product of a Schur function by these generating sets.

%%%%%%%%%%%%%%%%%%%%%%%%%%%%%%%%%%%%%%%%%%%%%%%%%%%%%%%%%%%%%%%%%%%%%%%%%%
\begin{proposition}\label{P:pieri_formulas}
 Let $\lambda$ be a partition and $a,b,r$ be positive integers.
 Then we have
\begin{enumerate}
 \item[$(i)$] ${\displaystyle e_a\cdot s_\lambda\ =\ \sum s_\mu}$, the sum over all
   $\mu$ with $\mu/\lambda$ a vertical strip of size $a$

 \item[$(ii)$]  ${\displaystyle h_b\cdot s_\lambda\ =\ \sum s_\mu}$, the sum over all
   $\mu$ with $\mu/\lambda$ a horizontal strip of size $b$.

 \item[$(iii)$]  ${\displaystyle p_r\cdot s_\lambda\ =\ 
               \sum (-1)^{\het(\mu/\lambda)+1}s_\mu}$, the sum over all
   $\mu$ with $\mu/\lambda$ a rim hook of size $r$. 
\end{enumerate}
\end{proposition}
%%%%%%%%%%%%%%%%%%%%%%%%%%%%%%%%%%%%%%%%%%%%%%%%%%%%%%%%%%%%%%%%%%%%%%%%%%

Formulas $(i)$ and $(ii)$ of Proposition~\ref{P:pieri_formulas} are called Pieri rules,
for they are essentially due to Pieri~\cite{Pi1893}, and Formula $(iii)$ is the
Murnaghan-Nakayama rule.

These formulas make sense in two important quotients of $\Lambda$.
The ring \defcolor{$\Lambda_k$} of polynomials that are symmetric in
$x_1,\dotsc,x_k$ for an integer $k>0$ is obtained from $\Lambda$ by specializing $x_a=0$
(equivalently $e_a=0$) for $a>k$.
This also has the vector space presentation
\[
  \Lambda_k\ \simeq\ \Lambda/\Q\{s_\lambda\mid \lambda_{k+1}>0\}\,.
\]
Thus $\Lambda_k$ has a basis of Schur polynomials $s_\lambda(x_1,\dotsc,x_k)$ where $\lambda$
has at most $k$ parts in that $\lambda_{k+1}=0$.
It follows that the Murnaghan-Nakayama rule holds in $\Lambda_k$, if we restrict 
to partitions with at most $k$ parts.

Let $n> k$ be an integer.
The \demph{Grassmannian $\Gr(k,n)$} of $k$-planes in $\C^{n}$ is a complex manifold of
dimension $k(n{-}k)$.
Its cohomology ring with $\Q$-coefficients is a quotient of both $\Lambda_k$ and $\Lambda$,
\[
   H^*(\Gr(k,n))\ \simeq\ 
   \Lambda/\langle e_a,h_b\mid a>k, b>n{-}k\rangle\ =\ 
   \Lambda/\Q\{ s_\lambda\mid \lambda\not\leq\Box_{k,n}\}\,,
\]
where $\Box_{k,n}$ is the partition with $k$ parts, each of size $n{-}k$.
We also have that 
 \begin{equation}\label{Eq:Grass_Quot}
  H^*(\Gr(k,n))\ \simeq\ 
   \Lambda_k/\langle h_{n{-}k{+}1},\dotsc,h_{n-1},h_{n}\rangle\,.
 \end{equation}

By these presentations, the images \defcolor{$\sigma_\lambda$} of the Schur functions $s_\lambda$ for
$\lambda\leq\Box_{k,n}$ form a basis for $H^*(\Gr(k,n))$.
These $\sigma_\lambda$ are called \demph{Schubert cycles} as they are Poincar\'e dual to
the fundamental cycles in homology of Schubert varieties.
The Pieri and Murnaghan-Nakayama rules hold in
$H^*(\Gr(k,n))$, if we restrict to partitions $\lambda,\mu\leq\Box_{k,n}$, and
to $a\leq k$, $b\leq n{-}k$, and $r< n$.

The homomorphism $\psi\colon \Lambda_k\twoheadrightarrow H^*(\Gr(k,n))$ implicit in the
quotient~\eqref{Eq:Grass_Quot} may be understood geometrically as follows.
The tautological bundle $E\to\Gr(k,n)$ is the subbundle of 
$\C^{n}\times\Gr(k,n)$ whose fibre over a point $H\in\Gr(k,n)$ is the $k$-plane $H$.
The map $\psi$ sends the elementary symmetric polynomial $e_a$ to the $a$th
Chern class $c_a(E^\vee)$ of the dual of $E$, equivalently, it sends the variables $x_1,\dotsc,x_k$ to
the Chern roots $y_1,\dotsc,y_k$ of $E^\vee$.

The Chern character $\defcolor{\ch(E^\vee)}\in H^*(\Gr(k,n))$ is the cohomology class
\[
   \ch(E^\vee)\ =\ \sum_{i=1}^k \exp(y_i)\ =\ 
    \sum_{i=1}^k 1+y_i + \frac{y_i^2}{2} + \frac{y_i^3}{3!} + \dotsb
\]
In terms of the map $\psi$, it is 
\[
   \psi(k + p_1 +\tfrac{1}{2}p_2+\tfrac{1}{3!}p_3 + \dotsb)\ =\ 
   k + \sum_{i=1}^{n-1} \tfrac{1}{i!}p_i(y_1,\dotsc,y_k) \,.
\]
The classes $\defcolor{\ch_r(E^\vee)}=\frac{1}{r!}p_r(y_1,\dotsc,y_k)$ are the \demph{tautological classes}.
Thus the Murnaghan-Nakayama rule is a formula in cohomology for multiplication by tautological classes.

%%%%%%%%%%%%%%%%%%%%%%%%%%%%%%%%%%%%%%%%%%%%%%%%%%%%%%%%%%%%%%%%%%%%%%%%%%
%
\section{Murnaghan-Nakayama rule for Schubert polynomials}\label{S:flag}

We prove Theorem~\ref{Th:one} using results about multiplication of Schubert polynomials.
An excellent reference for Schubert polynomials, with proofs, is Macdonald's book~\cite{Mac91}.

The symmetric group \demph{$S_n$} acts on polynomials in $x_1,\dotsc,x_n$ by permuting
the variables.
Write \defcolor{$(i,j)$} for the transposition interchanging $i$ and $j$ with $i<j$ and 
\defcolor{$t_i$} for the simple transposition $(i,i{+}1)$.
Permutations $w\in S_n$ are products of simple transpositions, $w=t_{a_1}\dotsb t_{a_m}$.  
The minimal length of such a factorization is when $m=\ell(w)$, the number of inversions in $w$.
In this case, $(a_1,\dotsc,a_m)$ is a \demph{reduced word} for $w$.

The divided difference operator for $i=1,\dotsc,n{-}1$ is
\[
    \defcolor{\partial_i}\ :=\ \frac{1-t_i}{x_i-x_{i+1}}\,,
\]
which acts on the polynomial ring $\Q[x_1,\dotsc,x_n]$.
These satisfy $\partial_i\circ\partial_i=0$ and the braid relations, so that if $(a_1,\dotsc,a_m)$ is a reduced word
for $w$, then $\partial_w:=\partial_{a_1}\circ\dotsb\circ\partial_{a_m}$ depends on $w$ and not on the choice of
reduced word.
Building on work of Bernstein, Gelfand, and Gelfand~\cite{BGG73} and of Demazure~\cite{De74}, Lascoux and
Sch\"utzenberger~\cite{LS82} defined the Schubert polynomial $\frakS_w$ to be
\[
    \defcolor{\frakS_w} \ :=\ \partial_{w^{-1}\omega_0}(x_1^{n-1}x_2^{n-2}\dotsb x_{n-1})\,,
\]
where \defcolor{$\omega_0$} is the longest element of the symmetric group $S_n$, $\omega_0(i)=n{+}1{-}i$.
Schubert polynomials are linearly independent.
For example, $\frakS_{t_k}=x_1+\dotsb+x_k$.

We may embed $S_n$ into $S_{n+m}$ as permutations fixing $n{+}1,\dotsc,n{+}m$.
Schubert polynomials are stable in the sense that 
for $w\in S_n$ we get the same the Schubert polynomial regarding $w\in S_n$ as we do regarding $w\in S_{m+n}$.
If \defcolor{$S_\infty$} is the set of permutations of $\{1,2,\dotsc,\}$ that fix all but finitely many integers, then
$\frakS_w$ is well-defined for $w\in S_\infty$, and these form a basis for the polynomial ring
$\Q[x_1,x_2,\dotsc]$.

All Schur symmetric polynomials are Schubert polynomials.
If $w$ is a permutation with a unique descent at $k$, so that $w(i)>w(i{+}1)$ implies that $i=k$, then 
$\frakS_w$ equals the Schur polynomial $s_\lambda(x_1,\dotsc,x_k)$, where
$\lambda = (w(k){-}k, \dotsc,w(2){-}2,w(1){-}1)$.

An open problem is to give a combinatorial formula for the expansion of a product of
Schubert polynomials in the basis of Schubert polynomials.
This is possible as the coefficients are nonnegative, and it would be an analog of the
Littlewood-Richardson rule. 
The first step was due to Monk~\cite{Monk},
 \begin{equation}\label{Eq:Monk}
   \frakS_{t_k}\cdot\frakS_w\ =\ (x_1+\dotsb+x_k)\cdot\frakS_w\ =\ 
    \sum \frakS_u\,,
 \end{equation}
the sum over all permutations $u=w(i,j)$ where $i\leq k<j$ and $\ell(u)=\ell(w)+1$.
The range of summation in~\eqref{Eq:Monk} gives the cover relation $\lessdot_k$ in the $k$-Bruhat order on the
symmetric group. 
Expanding $(x_1+\dotsb+x_k)^m$ in the basis of Schur polynomials and iterating the
formula~\eqref{Eq:Monk} shows that any Schubert polynomial $\frakS_u$ appearing in the
product of $\frakS_w$ with a symmetric polynomial in $x_1,\dotsc,x_k$ must have $w<_k u$.
When $w\lessdot_k w(i,j)$ with $i\leq k<j$, we write $w\xrightarrow{\,w(i)\,}w(i,j)$,
labeling the cover by the value of $w$ at $i$.

Monk's rule was generalized to a Pieri rule for multiplying by elementary or complete homogeneous symmetric
polynomials,  $e_a(x_1,\dotsc,x_k)$ and $h_b(x_1,\dotsc,x_k)$ in~\cite{So96}, and that paper also gave the formula for
hook Schur polynomials.

%%%%%%%%%%%%%%%%%%%%%%%%%%%%%%%%%%%%%%%%%%%%%%%%%%%%%%%%%%%%%%%%%%%%%%%%%%%%%%%%%
\begin{proposition}[\cite{So96}, Theorem 8] \label{Prop:peakless}
  Let $a \leq k$ and $b$ be positive integers and set $r = a+b-1$. 
  For a permutation $w\in S_\infty$, we have 
\[ 
    s_{(b,1^{a-1})} (x_1,\dotsc,x_k) \cdot\frakS_w\  =\
     \sum \frakS_{\ep(\gamma)} 
\]
the sum over all chains 
$\gamma \colon w \xrightarrow{\,\alpha_1\,} w^1\xrightarrow{\,\alpha_2\,}
      \dotsb \xrightarrow{\,\alpha_r\,} w^r=\ep(\gamma)$
in the $k$-Bruhat order satisfying 
 \begin{equation}\label{Eq:peakless}
   \alpha_1 > \dotsb > \alpha_a < \alpha_{a+1} < \dotsb < \alpha_{r}\,.
 \end{equation}
\end{proposition}
%%%%%%%%%%%%%%%%%%%%%%%%%%%%%%%%%%%%%%%%%%%%%%%%%%%%%%%%%%%%%%%%%%%%%%%%%%%%%%%%%
A chain $\gamma$ whose labels satisfy~\eqref{Eq:peakless} is \demph{peakless}.

Chains in the $k$-Bruhat order were studied in~\cite{BS98,BS99}.
For a permutation $\zeta$, let $\defcolor{\up(\zeta)}:=\{a\mid a<\zeta(a)\}$.
If $k$ is any integer with $k\geq|\up(\zeta)|$, then there is a permutation $w$ with $w<_k \zeta w$.
Moreover, the interval \defcolor{$[w,\zeta w]_k$} in the $k$-Bruhat order between $w$ and $\zeta w$, considered as a
poset whose covers are labeled, is independent of $w$ and $k$.
In particular, the \defcolor{rank} of $\zeta$ is  $\defcolor{\rk(\zeta)}:=\ell(\zeta w){-}\ell(w)$ for any
$w,k$ with $w<_k\zeta w$.
Two permutations $\zeta,\eta$ are \demph{disjoint} if $\zeta\eta=\eta\zeta$ and
$\rk(\zeta\eta)=\rk(\zeta)+\rk(\eta)$---this is equivalent to the supports 
of $\zeta$ and $\eta$ forming a noncrossing partition of their union
(see~\cite[\S~3.3]{BS98}).
By \demph{support} of a permutation $\zeta$, we mean the set $\{i\mid \zeta(i)\neq i\}$.

In~\cite[\S~6.2]{BS02} permutations $\zeta$ in which the interval $[w,\zeta w]_k$ admits a
peakless chain were studied. 
An $(r{+}1)$-cycle $\zeta$ is \demph{minimal} if it has the minimal possible rank $r$.
Lemma~6.7 of~\cite{BS02} states that if $\zeta$ is a minimal cycle and $w<_k \zeta w$, then there is a unique
peakless chain in $[w,\zeta w]_k$.
Moreover, $|\up(\zeta)|$ is the length of the decreasing subsequence in the labels of that chain~\eqref{Eq:peakless}, 
which is also equal to the \demph{height} of $\eta:=w^{-1}\zeta w$, which is 
$\defcolor{\het(\eta)}:=\#\{i\leq k\mid \eta(i)\neq i\}$.
For simplicity, we suppress the dependence of height on $k$, as $\eta$ itself depends upon
$w$, $k$, and $\zeta$ in this discussion.
We record these facts.

%%%%%%%%%%%%%%%%%%%%%%%%%%%%%%%%%%%%%%%%%%%%%%%%%%%%%%%%%%%%%%%%%%%%%%%%%%%%%%%%%
\begin{proposition}\label{P:minimal}
 Let $\zeta$ be a minimal cycle, $w$ be a permutation, and $k$ be an integer such that $w<_k\zeta w$.
 Then there is a unique peakless chain in $[w,\zeta w]_k$.
 If~$\eqref{Eq:peakless}$ is the sequence of labels in this chain and we set $\eta:=w^{-1}\zeta w$ so that 
 $\zeta w = w\eta$ then
\[
   |\up(\zeta)|\ =\ a\ =\ \het(\eta)\,.
\]
\end{proposition}
%%%%%%%%%%%%%%%%%%%%%%%%%%%%%%%%%%%%%%%%%%%%%%%%%%%%%%%%%%%%%%%%%%%%%%%%%%%%%%%%%
    
A permutation is minimal if it is the product of disjoint minimal cycles.
These results imply that a permutation $\zeta$ is minimal if
and only if whenever $w<_k \zeta w$, there is a peakless chain in the interval $[w,\zeta w]_k$.
We restate Theorem~\ref{Th:one}.\medskip

%%%%%%%%%%%%%%%%%%%%%%%%%%%%%%%%%%%%%%%%%%%%%%%%%%%%%%%%%%%%%%%%%%%%%%%%%%%%%%%%%
\noindent{\bf Theorem~\ref{Th:one}.\ }{\it
 Let $k,r$ be positive integers and $w\in S_\infty$ a permutation.
 Then
\[
    p_r(x_1,\dotsc,x_k)\cdot\frakS_w\ =\ 
   \sum (-1)^{\het(\eta)+1}\frakS_{w\eta}\,,
\]
 the sum over all $(r{+}1)$-cycles $\eta$ such that }
\[
    w\ <_k\ w\eta 
    \qquad\mbox{with}\qquad
    \ell(w\eta)=\ell(w)+r\,.
\]
%%%%%%%%%%%%%%%%%%%%%%%%%%%%%%%%%%%%%%%%%%%%%%%%%%%%%%%%%%%%%%%%%%%%%%%%%%%%%%%%%

%%%%%%%%%%%%%%%%%%%%%%%%%%%%%%%%%%%%%%%%%%%%%%%%%%%%%%%%%%%%%%%%%%%%%%%%%%%%%%%%%
\begin{example}
 Before presenting a proof of Theorem~\ref{Th:one}, we give an example of the Murnaghan-Nakayama rule for
 Schubert polynomials. 
 For this $k=r=4$, and we restrict to permutations $w\in S_8$, expressing such a $w$ as the word of its
 values $w(1)\dotsc w(8)$.
 \begin{eqnarray*}
   p_4(x_1,\dotsc,x_4)\cdot\frakS_{34165278}&=&
      \frakS_{35671248}+\frakS_{36471258}+\frakS_{45362178}+\frakS_{46173258}\\
    &&-\frakS_{34672158}-\frakS_{34681257}-\frakS_{36184257}\,.
 \end{eqnarray*}
 If $w=34165278$, then these indices are 
 \begin{eqnarray*}
  35671248 &=& w(\Blue{3},\Blue{4},7,\Blue{2},5)\,,\ 
  36471258\ =\ w(\Blue{2},\Blue{4},7,5,\Blue{3})\,,\ 
  45362178\ =\ w(\Blue{1},\Blue{2},5,6,\Blue{4})\,,\\ 
  46173258 &=& w(\Blue{1},\Blue{2},\Blue{4},7,5)\,,\ 
  34672158\ =\ w(\Blue{3},\Blue{4},7,5,6)\,,\ 
  34681257\ =\ w(\Blue{3},\Blue{4},8,7,5)\,,
 \end{eqnarray*}
  and $36184257\ =\ w(\Blue{2},\Blue{4},8,7,5)$.
  The first four have height 3, while the last three have height 2.
  This agrees with the signs in Theorem~\ref{Th:one}.
\end{example}
%%%%%%%%%%%%%%%%%%%%%%%%%%%%%%%%%%%%%%%%%%%%%%%%%%%%%%%%%%%%%%%%%%%%%%%%%%%%%%%%%

%%%%%%%%%%%%%%%%%%%%%%%%%%%%%%%%%%%%%%%%%%%%%%%%%%%%%%%%%%%%%%%%%%%%%%%%%%%%%%%%%
\begin{proof}[Proof of Theorem~$\ref{Th:one}$]
 Since $x_i=\frakS_{t_i}-\frakS_{t_{i-1}}$, Monk's rule gives the transition formula,
\[
   x_i\cdot\frakS_w\ =\ 
    \sum_{\substack{i< b\\ \ell(w(i,b))=\ell(w)+1}} \frakS_{w(i,b)}
   \ \ -\ \ 
    \sum_{\substack{a<i\\ \ell(w(a,i))=\ell(w)+1}} \frakS_{w(a,i)}\,.
\] 
 Thus if a Schubert polynomial $\frakS_u$ appears in 
 $x_i^r\cdot\frakS_w$, then $u=w\eta$, where $\eta$ is an $(r{+}1)$-cycle.
 Summing over $i=1,\dotsc,k$ we see that if $\frakS_u$ appears in 
 $p_r(x_1,\dotsc,x_k)\cdot\frakS_w$, then $u= w\eta$ with $\eta$ an $(r{+}1)$-cycle.
 As $p_r(x_1,\dotsc,x_k)$ is symmetric, we have $w<_kw\eta$.
 Writing $w\eta=\zeta w$, then $\zeta$ is an $(r{+}1)$-cycle of rank $r$,
 and is therefore a minimal cycle.

 By Proposition~\ref{P:minimal}, there is a unique peakless chain in $[w,\zeta w]_k$ and its sequence of decreasing
 labels has length $a:=\het(\eta)$.
 By Proposition~\ref{Prop:peakless}, the only product $s_{(r-i,1^i)}\cdot \frakS_w$ containing $\frakS_{w\eta}$
 is $s_{(r{-}a{+}1,1^{a{-}1})}\cdot \frakS_w$, and it occurs with multiplicity 1.
 By~\eqref{Eq:hookSum}, the coefficient of $\frakS_{w\eta}$ in $p_r\cdot\frakS_w$ is 
 $(-1)^{a-1}=(-1)^{\het(\eta)+1}$, which completes the proof. % of Theorem~\ref{Th:one}.
\end{proof}
%%%%%%%%%%%%%%%%%%%%%%%%%%%%%%%%%%%%%%%%%%%%%%%%%%%%%%%%%%%%%%%%%%%%%%%%%%%%%%%%%

%%%%%%%%%%%%%%%%%%%%%%%%%%%%%%%%%%%%%%%%%%%%%%%%%%%%%%%%%%%%%%%%%%%%%%%%%%%%%%%%%
\begin{remark}
 We apply the Murnaghan-Nakayama rule to the cohomology of flag manifolds.
 Fix a positive integer $n$.
 The flag manifold $\Fln$ is the collection of all complete flags, 
\[
   \Fdot\ =\ 0\ \subset\ F_1\ \subset\ F_2\ \subset \dotsb\ \subset\ F_n\ =\ \C^n\,,
\]
 where the subspace $F_i$ has dimension $i$.
 There is a universal family of flags over $\Fln$,
\[
   \calF_1\ \subset\ \calF_2\ \subset\ \dotsb\ \subset\ \calF_n\ =\ 
    \C^n\times\Fln\,,
\]
 where $\calF_i$ is the rank $i$ tautological subbundle given by the $i$th flag.
 For each $i=1,\dotsc,n$, let $\defcolor{y_i}:= -c_1(\calF_i/\calF_{i-1})$.
 These classes generate the cohomology of $\Fln$ with the only relations the non-constant symmetric polynomials
 in $y_1,\dotsc,y_n$.

 This has a cell decomposition.
 For a flag $\Edot\in\Fln$ and permutation $w\in S_n$, the set
\[
   \defcolor{X^\circ_w\Edot}\ :=\ \bigl\{\Fdot\in\Fln\mid  
     \dim(F_a\cap E_b)=\#\{i\leq a \mid w(i)\geq n{+}1{-}b\}\ \forall a,b\bigr\}\,,
\]
 is a topological cell of codimension $\ell(w)$.
 Writing \defcolor{$[X_w\Edot]$} for the cohomology class Poincar\'e dual to the closure of 
 $X^\circ_w\Edot$, these form a basis for the cohomology of $\Fln$,
\[
    H^*(\Fln)\ =\ \bigoplus_{w\in S_n} \Q [X_w\Edot]\,.
\]
 By~\cite{BGG73,De74,LS82} the homomorphism $\psi\colon\Q[x_1,x_2,\dotsc]\to H^*(\Fln)$ that sends 
 $x_i$ to $y_i$ when $i\leq n$ and to $0$ for $i>n$, sends the Schubert polynomial $\frakS_w$ to the Schubert cycle 
 $[X_w\Edot]$ and  $\frac{1}{r!}p_r(x_1,\dotsc,x_k)$ to the tautological class $\ch_r(\calF_k^\vee)$.
 Thus the Murnaghan-Nakayama rule for Schubert polynomials computes the intersection of Schubert cycles
 with tautological classes.

 Given $\adot\colon 0<a_1<\dotsb<a_s<n$, the partial flag manifold $\Flan$ consists of
 all flags
\[
   0\ \subset\ F_{a_1}\ \subset\ F_{a_2}\ \subset\ \dotsb\ \subset\ 
     F_{a_s}\ \subset \C^n\,,
\]
 where $\dim(F_{a_i})=a_i$.
 The cohomology of $\Flan$ is a subring of $H^*(\Fln)$.
 It has a basis of Schubert cycles $[X_w\Edot]$, where the descent set of $w$ is a 
 subset of $\{a_1,\dotsc,a_s\}$.
 Consequently, the Murnaghan-Nakayama rule for Schubert polynomials, when $k\in\{a_1,\dotsc,a_s\}$, computes the 
 intersection of Schubert cycles in $\Flan$ with tautological classes. 
\end{remark}
%%%%%%%%%%%%%%%%%%%%%%%%%%%%%%%%%%%%%%%%%%%%%%%%%%%%%%%%%%%%%%%%%%%%%%%%%%%%%%%%%

%%%%%%%%%%%%%%%%%%%%%%%%%%%%%%%%%%%%%%%%%%%%%%%%%%%%%%%%%%%%%%%%%%%%%%%%%%
%
\section{Murnaghan-Nakayama rule in the quantum cohomology of a
  Grassmannian}\label{S:quantum} 

Additively, the quantum cohomology ring $qH^*(\Gr(k,n))$ of the Grassmannian is equal
to $H^*(\Gr(k,n))[q]$, where $q$ is an indeterminate of degree $n$.
That is, it has a basis $q^d\sigma_\lambda$ for $d$ a nonnegative integer and
$\lambda\leq\Box_{k,n}$.
The product, \defcolor{$*$}, encodes three-point Gromov-Witten invariants~\cite{FP}. 
This was the first space whose quantum cohomology was computed~\cite{Be97,In91,ST97,Va92,Wi95}.
Its ring structure is given by
\[ 
   \Lambda_k[q]/\langle h_{n-k+1},\dotsc,h_{n-1},h_{n}+(-1)^kq\rangle
     \ \xrightarrow{\ \simeq\ }\  qH^*(\Gr(k,n))
\]
where $e_a$ is sent to $\sigma_{1^a}$, the $a$th Chern class of the dual of the tautological bundle.
If $q=0$, we recover the usual cohomology by~\eqref{Eq:Grass_Quot}.
Bertram~\cite{Be97} showed that the image of a Schur polynomial $s_\lambda$ for $\lambda\leq\Box_{k,n}$ is the
Schubert cycle $\sigma_\lambda$, just as in ordinary cohomology.
Thus $p_r(x_1,\dotsc,x_k)$ for $r<n$ represents tautological classes as in ordinary cohomology.

Quantum cohomology is also a quotient of $\Lambda_k$.
Sending $e_a$ to $\sigma_{1^a}$ as before and 
% the $a$th Chern class of the dual of the tautological bundle and 
$s_{(n{-}k{+}1,1^{k-1})}$ to the quantum parameter $q$ gives an isomorphism
 \begin{equation}\label{Eq:Qcoh_Pres}
  \Lambda_k/\langle h_{n-k+1},\dotsc,h_{n-1}\rangle \ 
    \xrightarrow{\ \sim\ }\ qH^*(\Gr(k,n))\,.
 \end{equation}
To see that this surjective map is an isomorphism we consider the identity
\[ 
   h_a -e_1h_{a-1} +e_2h_{a-2} -\cdots +(-1)^ae_a\ =\ 0
 \]
which specializes when $a=n$ to  $h_n +  (-1)^ks_{(n-k+1,1^{k-1})} = 0$
since $e_kh_{n-k} = s_{(n-k+1,1^{k-1})}$ in $\Lambda_k / ( h_{n-k+1},\dotsc , h_{n-1})$.

In their study of representations of Hecke algebras at $n$th roots of unity, 
Goodman and Wenzl~\cite{GW90} defined a (seemingly) different quotient of $\Lambda_k$,
 \begin{equation}\label{Eq:GW}
   \defcolor{\Lambda_{k,n}}\ :=\ 
    \Lambda_k/\langle s_\lambda\mid \lambda_1-\lambda_k=n{-}k{+}1\rangle\,,
 \end{equation}
and showed that the images of Schur polynomials $s_\lambda$ with
$\lambda_1{-}\lambda_k\leq n{-}k$ form a basis.
They obtained a formula for the corresponding Littlewood-Richardson coefficients 
with respect to this basis that was equal to a formula obtained by
Kac~\cite[Exer.\ 13.35]{Kac} and Walton~\cite{Wa90} for fusion coefficients in a
Wess-Zumino-Witten conformal field theory. 
Later, Bertram, Ciocan-Fontanine, and Fulton obtained the same formula for the
Littlewood-Richardson coefficients in the quantum cohomology of the
Grassmannian~\cite{BCF}, showing that $\Lambda_{k,n}$ is isomorphic to
$qH^*(\Gr(k,n))$ and that the ideals in~\eqref{Eq:Qcoh_Pres} and~\eqref{Eq:GW}
coincide.
Write $\psi$ for the map $\Lambda_k\twoheadrightarrow qH^*(\Gr(k,n))$ implicitly
defined by either~\eqref{Eq:Qcoh_Pres} or~\eqref{Eq:GW}.

Let $\lambda$ be a partition with at most $k$ parts.
The \demph{$n$-core} \defcolor{$\widehat{\lambda}$} of $\lambda$ is obtained from $\lambda$ by
removing rim hooks of size $n$ until it is not possible to remove any more.
The result is independent of choices~\cite{Ma95}.
When $k=4$ and $\lambda=(12,10,7,3)$, its $8$-core is $\widehat{\lambda}=(4,2,2,0)$,
and there are six different ways to remove $8$-hooks from $\lambda$.
\begin{center}
   \includegraphics{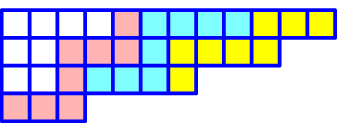}\qquad\rule{0pt}{35pt}
   \includegraphics{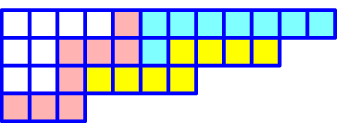}\qquad
   \includegraphics{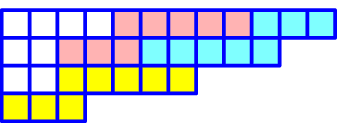}\vspace{10pt}\\
   \includegraphics{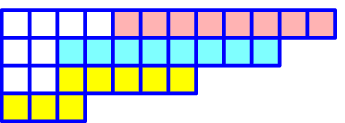}\qquad
   \includegraphics{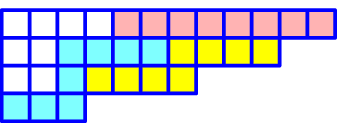}\qquad
   \includegraphics{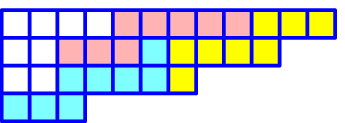}
\end{center}
Since $\Box_{k,n}$ has $n{-}1$ diagonals, it is impossible to remove a rim hook of size $n$ from any partition
$\lambda\leq\Box_{k,n}$.
Consequently, all partitions $\lambda\leq\Box_{k,n}$ are $n$-cores.

The image $\psi(s_\lambda)$ of a Schur polynomial $s_\lambda\in\Lambda_k$ in
$qH^*(\Gr(k,n))$ was given in~\cite{BCF} by  
 \begin{equation}\label{Eq:image}
   s_\lambda\ \longmapsto\ 
    \left\{\begin{array}{rcl}
       (-1)^{ks-\sum\het(\rho_i)} q^s \sigma_{\widehat{\lambda}}
        &\ &\mbox{if } \widehat{\lambda}\leq \Box_{k,n}\\
        0&&\mbox{otherwise}\end{array}\right.\ ,
 \end{equation}
where we remove $s$ rim hooks $\rho_1,\dotsc,\rho_s$ of size $n$ from $\lambda$ to obtain its $n$-core
$\widehat{\lambda}$, 
Thus $\psi(s_{(12,10,7,3)})=q^3 s_{(4,2,2)}$, as $k=4$ and the sum of the heights of the three hooks is even.
Similarly, $\psi(s_{(9,8,5,2)})=0$, as the $8$-core of $(9,8,5,2)$ is $(7,4,3,2)$,  
and $(7,4,3,2)\not\leq\Box_{4,8}$.\medskip

%%%%%%%%%%%%%%%%%%%%%%%%%%%%%%%%%%%%%%%%%%%%%%%%%%%%%%%%%%%%%%%%%%%%%%%%%%%%%%%%%
\noindent{\bf Theorem~\ref{Th:two}.} {\it
 Let $k,r< n$ be positive integers and $\lambda\leq \Box_{k,n}$.
 Then
 \begin{eqnarray}\label{Eq:q-MN}
    p_r * \sigma_\lambda\ =\ 
      \sum_\mu  (-1)^{\het(\mu/\lambda)+1} \sigma_\mu 
     \ -\ (-1)^k q \sum_\nu  (-1)^{\het(\lambda/\nu)+1} \sigma_\nu\,,
 \end{eqnarray}
 where the first sum is over all $\mu\leq\Box_{k,n}$ with $\mu/\lambda$ a rim hook of size $r$ and the second
 sum is over all $\nu\leq\lambda$ with $\lambda/\nu$ a rim hook of size $n{-}r$.}\medskip
%%%%%%%%%%%%%%%%%%%%%%%%%%%%%%%%%%%%%%%%%%%%%%%%%%%%%%%%%%%%%%%%%%%%%%%%%%%%%%%%%

Suppose that $k=4$, $n=8$, and $r=5$, then 
 \begin{equation}\label{Eq:q-MN_ex}
   p_5 * \sigma_{(3,2,1)}\ =\ \sigma_{(3,3,3,2)} + \sigma_{(4,4,3)} 
       + q\sigma_{(3)} + q \sigma_{(1,1,1)}\,.
 \end{equation}
To see this, first consider all four ways of adding a rim hook of size 5 to $(3,2,1)$,
\[
   \includegraphics{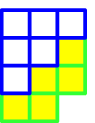}\qquad
   \includegraphics{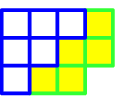}\qquad
   \includegraphics{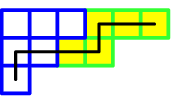}\qquad
   \includegraphics{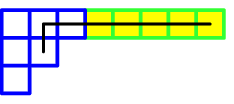}
\]
For the last two, whose first row exceeds $8{-}4$, we indicate their unique rim hooks of size 8.
Removing them gives partitions $\nu\leq\lambda$ with $\lambda/\nu$ a rim hook of size $3=8-5$.
Applying the formulas for the signs in~\eqref{Eq:q-MN} gives the expression~\eqref{Eq:q-MN_ex}.

%%%%%%%%%%%%%%%%%%%%%%%%%%%%%%%%%%%%%%%%%%%%%%%%%%%%%%%%%%%%%%%%%%%%%%%%%%%%%%%%%
\begin{proof}
 Let $\lambda\leq\Box_{k,n}$ and $1\leq r<n$.
 The Murnaghan-Nakayama rule in $\Lambda_k$ is
 \begin{equation}\label{Eq:Class_MN}
     p_r \cdot s_\lambda\ =\ 
      \sum_\mu  (-1)^{\het(\mu/\lambda)+1} s_\mu \,,
 \end{equation}
 the sum over all partitions $\mu$ with $k$ parts such that $\mu/\lambda$ is a rim hook of size $r$.
 We apply the map $\psi$ to~\eqref{Eq:Class_MN} to obtain a formula for $p_r * \sigma_\lambda$.
 Terms of the sum~\eqref{Eq:Class_MN} indexed by partitions $\mu\leq\Box_{k,n}$ contribute to the first, classical,
 sum in~\eqref{Eq:q-MN}, as such $\mu$ are $n$-cores.
 Suppose that $\mu\not\leq\Box_{k,n}$ indexes a term in the sum~\eqref{Eq:Class_MN}.
 Let $\nu$ be the $n$-core of $\mu$.
 If $\nu\not\leq\Box_{k,n}$, then $\psi(s_\mu)=0$.
 If $\nu\leq\Box_{k,n}$, then we will show (1) that $\mu/\nu$ is a rim hook of size $n$, 
 (2) that $\nu\leq\lambda$ with $\lambda/\nu$ is a rim hook of size $n{-}r$, and 
 (3) that 
\[
   \het(\lambda/\nu)\ +\ \het(\mu/\lambda)\ =\ \het(\mu/\nu)\ +\ 1\,.
\]
 These together imply that 
 \begin{equation}\label{Eq:quantum-signs}
    \psi\bigl( (-1)^{\het(\mu/\lambda)+1} s_\mu \bigr)
   \ =\ q(-1)^{\het(\mu/\nu)+1+k-\het(\mu/\nu)} \sigma_\nu
   \ =\ -(-1)^k q (-1)^{\het(\lambda/\nu)+1} \sigma_\nu\,.
 \end{equation}
 We complete the proof by showing (4) that all partitions $\nu\leq\lambda$ with $\lambda/\nu$ a
 rim hook of size $n{-}r$ arise in this way.

 A rim hook is determined by either its Southwestern-most or Northeastern-most box.
 If we have a rim hook of size $t$ in a partition $\kappa$ with $k$ parts whose Northeastern-most box
 is the last box in row $i$ of $\kappa$, then that rim hook consists of the $t$ consecutive boxes along the rim of
 $\kappa$ starting from the end of row $i$ and moving Southwestward.
 Necessarily, $\kappa_i{+}k{-}i\geq r$, as there are only $\kappa_i{+}k{-}i$ diagonals Southwest of this first box.

 Suppose that $\mu\not\leq\Box_{k,n}$ is obtained from $\lambda\leq\Box_{k,n}$ by adding a rim hook of size
 $r$.
 Then the Northeastern-most box of $\mu/\lambda$ is in the first row of $\mu$.
 Moreover, we either have that $\mu_2=\lambda_1{+}1$ or $\mu_2=\lambda_2$, depending on whether or not the rim hook
 has height exceeding 1.
 In either case, it is not possible to remove a rim hook of size $n$ from $\mu$ starting from any row other
 than 1. 

 Consequently, if we remove a rim hook of size $n$ from $\mu$ to obtain a partition $\nu$, we necessarily first
 remove the rim hook $\mu/\lambda$, and then a further $n{-}r$ boxes from $\lambda$.
 This implies that $\nu$ satisfies $\nu\leq\lambda$, and so it is an $n$-core, which implies
 claims (1) and (2).
 The rim hook $\lambda/\nu$ has Northeastern-most box in the same row as the Southwestern-most box of
 $\mu/\lambda$, which implies claim (3) and~\eqref{Eq:quantum-signs}.

 For the last claim, as $\Box_{k,n}$ contains $n{-}1$ diagonals, we may add a rim hook of size $n$ to any
 partition $\kappa\leq\Box_{k,n}$ starting in any row $i\leq k$, obtaining a partition
 $\kappa'\not\leq\Box_{k,n}$.
 Thus if $\nu\leq\lambda$ is any partition with $\lambda/\nu$ a rim hook of size $n{-}r$, we may extend this rim
 hook rightward to a rim hook of size $n$.
 If $\mu$ is the partition obtained, then $\mu\not\leq\Box_{k,n}$, $\mu/\lambda$ is a rim hook of size $r$, 
 $\mu/\nu$ is a rim hook of size $n$, and $\nu$ is the $n$-core of $\mu$.
\end{proof}

Theorem~2 can be extended to multiplication by the image $\psi(p_r)$ for
$r\geq n$, as in that case $\psi(p_r)=(-1)^kq\psi(p_{r-n})$, which may be seen
using~\eqref{Eq:hookSum} and~\eqref{Eq:image}.
%%%%%%%%%%%%%%%%%%%%%%%%%%%%%%%%%%%%%%%%%%%%%%%%%%%%%%%%%%%%%%%%%%%%%%%%%%%%%%%%%
\providecommand{\bysame}{\leavevmode\hbox to3em{\hrulefill}\thinspace}
\providecommand{\MR}{\relax\ifhmode\unskip\space\fi MR }
\providecommand{\MRhref}[2]{%
  \href{http://www.ams.org/mathscinet-getitem?mr=#1}{#2}
}
\providecommand{\href}[2]{#2}

\end{document}